\documentclass[12pt]{article}
\usepackage{amsfonts}
\usepackage{amsmath}
\usepackage{amssymb}
\usepackage{amsmath}
\usepackage{epsfig}
\usepackage{graphicx}
\usepackage{xcolor}

\newtheorem{theorem}{Theorem}[section]

\newtheorem{corollary}{Corollary}[section]

\newtheorem{example}{Example}[section]

\newtheorem{lemma}{Lemma}[section]

\newtheorem{proposition}[theorem]{Proposition}
\newtheorem{remark}{Remark}[section]

\newenvironment{proof}[1][Proof]{\noindent\textbf{#1.} }{\ \rule{0.5em}{0.5em}}

\begin{document}

\title{On universal realizability of spectra.\thanks{%
Supported by Fondecyt 1170313, Chile; Conicyt-PAI 79160002, 2016, Chile;
MTM2015-365764-C-1 (MINECO/FEDER); MTM2017-85996-R (MINECO/FEDER); Consejer\'ia de Educaci\'on de la Junta de Castilla y Le\'on (Spain) VA128G18. }}
\author{Ana I. Julio$^{a}$, Carlos Mariju\'{a}n$^{b},$ Miriam Pisonero$^{b},$
Ricardo L. Soto$^{a}$\thanks{%
*Corresponding author: rsoto@ucn.cl (R.L. Soto), ajulio@ucn.cl (A.I. Julio),
marijuan@mat.uva.es (C. Mariju\'{a}n), mpisoner@maf.uva.es (M. Pisonero).} \\
$^{a}${\small Dpto. Matem\'{a}ticas, Universidad Cat\'{o}lica del Norte,
Casilla 1280}\\
{\small Antofagasta, Chile.}\\
$^{b}${\small Dpto. Matem\'{a}tica Aplicada, Universidad de Valladolid/IMUVA, Spain. }}
\date{}
\maketitle

\begin{abstract}
A list $\Lambda =\{\lambda _{1},\lambda_{2},\ldots ,\lambda _{n}\}$ of complex numbers
is said to be \textit{realizable} if it is the spectrum of an entrywise nonnegative matrix. The list $\Lambda $ is said to be \textit{universally realizable} ($\mathcal{%
UR}$) if it is the spectrum of a nonnegative matrix for each possible Jordan
canonical form allowed by $\Lambda $. It is well known that an $n\times n$
nonnegative matrix $A$ is co-spectral to a nonnegative matrix $B$ with
constant row sums. In this paper, we extend the co-spectrality between $A$
and $B$ to a similarity between $A$ and $B$, when the Perron eigenvalue is simple.
 We also show that if \ $\epsilon \geq 0$ and $\Lambda =\{\lambda _{1},\lambda_{2},\ldots ,\lambda _{n}\}$ is $\mathcal{UR}%
,$ then $\{\lambda _{1}+\epsilon ,\lambda _{2},\ldots
,\lambda _{n}\}$ is also $\mathcal{UR}$.  We give
counter-examples for the cases: $\Lambda =\{\lambda
_{1},\lambda_{2},\ldots ,\lambda _{n}\}$ is $\mathcal{UR}$ implies $\{\lambda _{1}+\epsilon ,\lambda _{2}-\epsilon ,\lambda_{3},\ldots ,\lambda
_{n}\}$ is $\mathcal{UR},$ and $\Lambda _{1},\Lambda _{2}$ are $\mathcal{UR}$
implies $\Lambda _{1}\cup \Lambda _{2}$ is  $\mathcal{UR}$.
\end{abstract}


\section{Introduction}

Let $M_{n}$ denote the set of $n\times n$ real matrices and $M_{k,l}$ the set of $k\times l$ real matrices. Let $A\in M_{n}$ and let 
\begin{equation*}
J(A)=S^{-1}AS=diag\left( J_{n_{1}}(\lambda _{1}),J_{n_{2}}(\lambda
_{2}),\ldots ,J_{n_{k}}(\lambda _{k})\right)
\end{equation*}%
be the \textit{Jordan canonical form} of $A$ (hereafter JCF of $A$), where
the $n_{i}\times n_{i}$ submatrices 
\begin{equation*}
J_{n_{i}}(\lambda _{i})=%
\begin{bmatrix}
\lambda _{i} & 1 &  &  \\ 
& \lambda _{i} & \ddots &  \\ 
&  & \ddots & 1 \\ 
&  &  & \lambda _{i}%
\end{bmatrix}%
,\text{ }i=1,\ldots ,k,
\end{equation*}%
are called the \textit{Jordan blocks} of $J(A)$. The \textit{elementary
divisors} of $A$ are the characteristic polynomials of $J_{n_{i}}(\lambda
_{i})$, $i=1,\ldots ,k$. The \textit{nonnegative inverse elementary divisors
problem} (hereafter NIEDP) is the problem of determining necessary and
sufficient conditions for the existence of an $n\times n$ entrywise nonnegative matrix
with prescribed elementary divisors \cite{Collao, Diaz Soto, Minc 1, Minc 2,
Soto Ccapa, Soto Diaz, Soto Julio, Soto4}. If there exists a nonnegative
matrix with spectrum $\Lambda =\{\lambda _{1},\lambda_{2},\ldots ,\lambda _{n}\}$ for
each possible Jordan canonical form allowed by $\Lambda ,$ we say that $%
\Lambda $ is \textit{universally realizable} $(\mathcal{UR}).$ If $\Lambda $ is the
spectrum of a nonnegative diagonalizable matrix, then $\Lambda $ is said to
be \textit{diagonalizably realizable} $(\mathcal{DR}).$

The NIEDP is closely related to the \textit{nonnegative inverse
eigenvalue problem} (hereafter NIEP), which is the problem of characterizing
all possible spectra of entrywise nonnegative matrices. If there is a
nonnegative matrix $A$ with spectrum $\Lambda =\{\lambda _{1},\lambda_{2},\ldots
,\lambda _{n}\},$ we say that $\Lambda $ is \textit{realizable} and that $A$ is a
\textit{realizing matrix}. Both problems, the NIEDP and the NIEP, remain unsolved. A
complete solution for the NIEP is known only for $n\leq 4.$

Throughout this paper, the first written element of a list $\Lambda
=\{\lambda _{1},\lambda_{2},\ldots ,\lambda _{n}\},$ \textit{i.e.} $\lambda _{1},$ is the Perron
eigenvalue of $\Lambda ,$ $\lambda _{1}=\max \{\left\vert \lambda
_{i}\right\vert ,$ $\lambda _{i}\in \Lambda \}.$ If $\Lambda $ is the
spectrum of a nonnegative matrix $A,$ we write $\rho (A)=\lambda _{1}$ for
the spectral radius of $A$.

In this paper, we ask whether certain properties of the NIEP, such
as the three rules that characterize the $C$-realizability of lists (see 
\cite{Borobia}), extend or not to the NIEDP. In particular, we ask: \newline
$1)$ If $\Lambda =\{\lambda _{1},\lambda_{2},\ldots ,\lambda _{n}\}$ is $\mathcal{UR}$,
is $\{\lambda _{1}+\epsilon ,\lambda _{2},\ldots ,\lambda _{n}\}$ also $%
\mathcal{UR}$ for any $\epsilon >0$?\newline
$2)$ If $\Lambda =\{\lambda _{1},\lambda _{2},\ldots ,\lambda _{n}\}$ is $%
\mathcal{UR}$ and $\lambda _{2}$ is real, is $\{\lambda _{1}+\epsilon
,\lambda _{2}- \epsilon ,\lambda _{3},\ldots ,\lambda _{n}\}$ also $%
\mathcal{UR}$ for any $\epsilon >0$?\newline
$3)$ If the lists $\Lambda _{1}$ and $\Lambda _{2}$ are $\mathcal{UR}$, is $%
\Lambda _{1}\cup \Lambda _{2}$ also $\mathcal{UR}$?

In \cite{Cronin}, Cronin and Laffey examine the subtle difference
between the \textit{symmetric nonnegative inverse eigenvalue problem} (SNIEP), in
which the realizing matrix is required to be symmetric, and the \textit{real
diagonalizable nonnegative inverse eigenvalue problem} (DRNIEP), in which the
realizing matrix is diagonalizable. The authors in \cite{Cronin} give
examples of lists of real numbers, which can be the spectrum of a
nonnegative matrix, but not the spectrum of a diagonalizable nonnegative
matrix.

The set of all $n\times n$ real matrices with constant row sums
equal to $\alpha \in 
\mathbb{R}
$ will be denoted by $CS_{\alpha }.$ It is clear that $\mathbf{e}%
=[1,1,\ldots ,1]^{T}$ is an eigenvector of any matrix $A\in CS_{\alpha },$
corresponding to the eigenvalue $\alpha .$ Denote by $\mathbf{e}_{k}$ the
vector with $1$ in the $k^{th}$ position and zeros elsewhere. The importance
of matrices with constant row sums is due to the well known fact that an $%
n\times n$ nonnegative matrix $A$ with spectrum $\Lambda =\{\lambda
_{1},\lambda_{2},\ldots ,\lambda _{n}\},$ $\lambda _{1}$ being the Perron eigenvalue, is
co-spectral to a nonnegative matrix $B\in \mathcal{CS}_{\lambda _{1}}$ \cite%
{Johnson, Guo}. In this paper, we extend the co-spectrality between $A$ and $%
B$ to similarity between $A$ and $B,$ when $\lambda _{1}$ is simple, and
therefore $J(A)=J(B).$ In what follows, we use the following notations and
results: we write $A\geq 0$ if $A$ is a nonnegative matrix, and $A>0$ if $A$
is a positive matrix, that is, if all its entries are positive. We shall use the same notation for vectors.

\begin{theorem}{\rm \cite[(2.7) Theorem p. 141]{Berman} 
\label{Berman}}  Let $A\in \{M=(m_{ij})\in M_{n} : m_{ij}\leq 0, i\neq j\}$ be an irreducible matrix. Then each one of the following conditions is equivalent to the statement: ``$%
A $ is a nonsingular $M$-matrix". \newline
$i)$ $A^{-1}$ is positive. \newline
$ii)$ $A\mathbf{x}\geq 0$ and $A\mathbf{x}\neq 0$ for some $\mathbf{x}$ positive.
\end{theorem}

\begin{theorem}{\rm \cite{Soto Ccapa}\label{Soto Ccapa}}
Let $\mathbf{q}=[q_{1},\ldots
,q_{n}]^{T} $ be an arbitrary $n$-dimensional vector and $E_{11}\in M_{n}$
with $1$ in the $(1,1)$ position and zeros elsewhere. Let $A\in \mathcal{CS}%
_{\lambda _{1}}$ with JCF 
\begin{equation*}
J(A)=S^{-1}AS=diag\left( J_{1}(\lambda _{1}),J_{n_{2}}(\lambda _{2}),\ldots
,J_{n_{k}}(\lambda _{k})\right) .
\end{equation*}%
If $\lambda _{1}+\sum_{i=1}^{n}q_{i}\neq \lambda _{i},$ \ $i=2,\ldots ,n$,
then the matrix $A+\mathbf{eq}^{T}$ has Jordan canonical form $%
J(A)+(\sum_{i=1}^{n}q_{i})E_{11}$. In particular, if $\sum_{i=1}^{n}q_{i}=0,$
then $A$ and $A+\mathbf{eq}^{T}$ are similar.
\end{theorem}

This paper is organized as follows: In Section $2,$ we extend the
co-spectrality between a nonnegative matrix $A$ and a nonnegative matrix $B$
with constant row sums to a similarity between $A$ and $B,$ when the Perron
eigenvalue is simple. In Section $3,$ we show that if a list of complex
numbers $\Lambda =\{\lambda _{1},\lambda_{2},\ldots ,\lambda _{n}\}$ is $\mathcal{UR},$
then $\{\lambda _{1}+\epsilon ,\lambda _{2},\ldots
,\lambda _{n}\}$ is also $\mathcal{UR}$ for any $\epsilon >0.$ We also
consider the universal realizability of the Guo perturbation $\{\lambda _{1}+\epsilon ,\lambda _{2}-\epsilon
,\lambda _{3},\ldots ,\lambda _{n}\},$ and of the union of two universally
realizable lists $\Lambda _{1}$ and $\Lambda _{2}.$ In Section $4,$ we study the nonsymmetric realizablity of lists of size $5$ with trace zero and three negative elements.

\section{Nonnegative matrices similar to nonnegative matrices with constant row sums}

\noindent It is well known that if $A$ is an irreducible nonnegative matrix,
then $A$ has a positive eigenvector associated to its Perron eigenvalue. In
this section, we extend this result to reducible matrices under certain
conditions. As a consequence, in both cases, $A$ is similar to a nonnegative
matrix $B$ with constant row sums when the Perron eigenvalue is simple. In
this way, we extend a result attributed to Johnson \cite{Johnson},
about the co-spectrality between a nonnegative matrix $A$ and a nonnegative
matrix $B\in \mathcal{CS}_{\lambda _{1}}.$

\begin{lemma}
\label{lema1} Let $A\in M_{n}$ be a nonnegative matrix of the form 
\begin{equation*}
A=%
\begin{bmatrix}
A_{1} & 0 \\ 
A_{3} & A_{2}%
\end{bmatrix}%
,
\end{equation*}%
with $A_{1}\in \mathcal{CS}_{\lambda _{1}}$, $A_{3}\neq 0,$ $A_{2}$
irreducible and $\lambda _{1}=\rho (A)=\rho (A_{1})>\rho (A_{2})$. Then $A$
has a positive eigenvector associated to $\lambda _{1}$. Moreover, there
exists a nonnegative matrix $B\in \mathcal{CS}_{\lambda _{1}}$ similar to $%
A. $
\end{lemma}
\begin{proof}
Let $A_{1}\in M_{k}$ and $A_{2}\in M_{n-k}$. Let $\mathbf{x}=%
\begin{bmatrix}
\mathbf{e} \\ 
\mathbf{y}%
\end{bmatrix}%
$ with $\mathbf{e}\in M_{k,1},$ $\mathbf{y}\in M_{n-k,1}$. Then, for%
\begin{equation*}
\begin{bmatrix}
A_{1} & 0 \\ 
A_{3} & A_{2}%
\end{bmatrix}%
\begin{bmatrix}
\mathbf{e} \\ 
\mathbf{y}%
\end{bmatrix}%
=%
\begin{bmatrix}
A_{1}\mathbf{e} \\ 
A_{3}\mathbf{e}+A_{2}\mathbf{y}%
\end{bmatrix}%
=%
\begin{bmatrix}
\lambda _{1}\mathbf{e} \\ 
\lambda _{1}\mathbf{y}%
\end{bmatrix}%
,
\end{equation*}\ \\
we have \ $A_{3}\mathbf{e}=(\lambda _{1}I-A_{2})\mathbf{y},$ where $\lambda _{1}I-A_{2}$ is an irreducible nonsingular $M$-matrix. Then,
from Theorem \ref{Berman}, $(\lambda _{1}I-A_{2})^{-1}>0$. Therefore, 
\begin{equation}
\mathbf{y}=(\lambda _{1}I-A_{2})^{-1}(A_{3}\mathbf{e})>0,  \label{y}
\end{equation}%
and so $\mathbf{x}^{T}=\left[ \mathbf{e}^{T},\mathbf{y}^{T}\right] =\left[
x_{1},\ldots ,x_{n}\right] $ is positive$.$ Then, for $D=diag\left(
x_{1},\ldots ,x_{n}\right) ,$ $B=D^{-1}AD$ is similar to $A.$ Since 
\begin{equation*}
B\mathbf{e}=D^{-1}AD\mathbf{e}=\lambda _{1}\mathbf{e,}
\end{equation*}%
then $B\in \mathcal{CS}_{\lambda _{1}}.$
\end{proof}

\begin{remark} Note that the eigenvector $\mathbf{x}$ obtained in the proof of Lemma \ref{lema1} is  $\mathbf{x}^{T}=[\mathbf{e}^{T},\mathbf{y}^{T}]$, where $\mathbf{e}$ has the number of rows $A_{1}$ and \begin{align*}\mathbf{y}=(\lambda_{1}I-A_{2})^{-1}(A_{3}\mathbf{e})=[y_{1},\ldots,y_{n-k}]^{T}>0.\end{align*} Let $Y=diag(y_{1},\ldots,y_{n-k})$, then a matrix $B\in \mathcal{CS}_{\lambda_{1}}$ similar to $A$ is of the form \[B=\begin{bmatrix}A_{1}& 0 \\ Y^{-1}A_{3} & Y^{-1}A_{2}Y\end{bmatrix}.\] 
\end{remark}

Note that in Lemma \ref{lema1} it is not necessary that the
spectral radius of $A$ be simple, as shown in matrix%
\begin{equation*}
A=\left[\begin{array}{cc|c}2&0&0\\0&2&0\\\hline2&0&1\end{array}\right],
\end{equation*}%
which has a positive eigenvector $\left[ 1,1,2\right]$ associated to
the double eigenvalue $\lambda _{1}=2$.

Now, suppose that $A$ is a block
diagonal matrix. Then, for this case, we have the following result:

\begin{lemma}
\label{lema2} Let $A\in M_{n}$ be a nonnegative matrix of the form 
\begin{equation*}
A=%
\begin{bmatrix}
A_{1} & 0 \\ 
0 & A_{2}%
\end{bmatrix}%
,
\end{equation*}%
with $A_{1}\in \mathcal{CS}_{\lambda _{1}},$ $A_{2}$ irreducible and $%
\lambda _{1}=\rho (A)=\rho (A_{1})>\rho (A_{2}).$ Then $A$ is similar to a
nonnegative matrix $\widetilde{A}=%
\begin{bmatrix}
A_{1} & 0 \\ 
A_{3} & A_{2}%
\end{bmatrix}%
$, with $A_{3}\neq 0.$ Moreover, there exists a nonnegative matrix $B\in 
\mathcal{CS}_{\lambda _{1}}$ similar to $A.$
\end{lemma}

\begin{proof}
Let $A_{1}\in M_{k}$ and $A_{2}\in M_{n-k}$. We suppose, without loss of
generality, that $A_{2}\in \mathcal{CS}_{\rho (A_{2})}$. Define the
nonsingular matrix%
\begin{equation*}
S=%
\begin{bmatrix}
I_{k} & 0 \\ 
-Z & I_{n-k}%
\end{bmatrix}%
,\ \ \text{with}\ \ S^{-1}=%
\begin{bmatrix}
I_{k} & 0 \\ 
Z & I_{n-k}%
\end{bmatrix}%
,
\end{equation*}%
where $Z=\mathbf{e}\mathbf{z}^{T}\in M_{n-k,k}$, with $\mathbf{z}$ being an
eigenvector of $A_{1}^{T}$ associated to $\lambda _{1}$. Then 
\begin{equation*}
\widetilde{A}=S^{-1}AS=%
\begin{bmatrix}
A_{1} & 0 \\ 
ZA_{1}-A_{2}Z & A_{2}%
\end{bmatrix}%
.
\end{equation*}%
We show that $A_{3}=ZA_{1}-A_{2}Z$ is a nonzero nonnegative matrix. The
entry in position $(r,j)$ of the matrix $A_{3}$ is, 
\begin{align*}
\mathbf{e}_{r}^{T}(ZA_{1}-A_{2}Z)\mathbf{e}_{j}& =\mathbf{z}%
^{T}col_{j}(A_{1})-z_{j}row_{r}(A_{2})\mathbf{e} \\
& =\sum_{i=1}^{k}a_{ij}z_{i}-z_{j}\rho (A_{2}),
\end{align*}%
for all $r=1,\ldots ,n-k,$ $j=1,\ldots ,k.$ Therefore, $ZA_{1}-A_{2}Z$ has
all its rows equal, which can be expressed as 
\begin{equation}
(A_{1}^{T}-\rho (A_{2})I_{k})\mathbf{z}.  \label{z1}
\end{equation}%
Since $A_{1}^{T}-\rho (A_{2})I_{k}$ and $A_{1}^{T}$ have the same
eigenvectors, then from (\ref{z1}) 
\begin{align*}
(A_{1}^{T}-\rho (A_{2})I_{k})\mathbf{z}& =A_{1}^{T}\mathbf{z}-\rho (A_{2})%
\mathbf{z} \\
& =\lambda _{1}\mathbf{z}-\rho (A_{2})\mathbf{z} \\
& =(\lambda _{1}-\rho (A_{2}))\mathbf{z}\geq 0.
\end{align*}%
Therefore $A_{3}=ZA_{1}-A_{2}Z$ is a nonzero nonnegative matrix. Since $A$
and $\widetilde{A}$ are similar with $A_{3}$ nonzero nonnegative, then from
Lemma \ref{lema1} there exists a nonnegative matrix $B\in \mathcal{CS}%
_{\lambda _{1}}$ similar to $A.$
\end{proof}

\begin{remark}
Note that the matrix $A_{3}$ in the proof of Lemma \ref{lema2} is \begin{align}\label{A3}A_{3}=\mathbf{ez}^{T}A_{1}-A_{2}\mathbf{ez}^{T},\end{align}
with $\mathbf{z}$ being an eigenvector of $A_{1}^{T}$ associated to $\lambda_{1}$. Then, from Lemma \ref{lema1}, $\widetilde{A}$ has a  positive eigenvector $\mathbf{x}=[\mathbf{e}^{T},\mathbf{y}^{T}]$ associated to $\lambda_{1}$, where \begin{align*}\mathbf{y}=(\lambda_{1}I-A_{2})^{-1}(A_{3}\mathbf{e})=[y_{1},\ldots,y_{n-k}]^{T}, \ \text{with} \ A_{3} \ \text{as in \eqref{A3}}. \end{align*} Let $Y=diag\{y_{1},\ldots,y_{n-k}\}$, then a matrix $B\in \mathcal{CS}_{\lambda_{1}}$ similar to $\widetilde{A}$ is of the form \[B=\begin{bmatrix}A_{1}& 0 \\ Y^{-1}A_{3} & Y^{-1}A_{2}Y\end{bmatrix}.\]  
\end{remark}

Next we prove the main result in this section. This result extends
the co-spectrality between a nonnegative matrix $A$ and a nonnegative matrix 
$B\in \mathcal{CS}_{\lambda _{1}},$ to a similarity between $A$ and $B.$

\begin{theorem}
\label{Teorema} Let $A\in M_{n}$ be a nonnegative matrix with $\lambda
_{1}=\rho (A)$ simple. Then there exists a nonnegative matrix $B\in \mathcal{%
CS}_{\lambda _{1}}$ similar to $A.$
\end{theorem}

\begin{proof}
If $A$ is irreducible, then $A$ has a positive eigenvector $\mathbf{x}=\left[
x_{1},\ldots ,x_{n}\right] ^{T}$ associated to $\lambda _{1}\mathbf{.}$ Let $%
D=diag\left( x_{1},\ldots ,x_{n}\right) $. Then $B=D^{-1}AD\in \mathcal{CS}%
_{\lambda _{1}}$ is nonnegative and similar to $A$.

If $A$ is reducible, then $A$ is permutationally similar to 
\begin{equation*}
\widetilde{A}=%
\begin{bmatrix}
A_{11} &  &  &  &  &  &  \\ 
A_{21} & A_{22} &  &  &  &  &  \\ 
\vdots  & \ddots  & \ddots  &  &  &  &  \\ 
A_{k1} & \cdots  & A_{k,k-1} & A_{kk} &  &  &  \\ 
0 & \cdots  & 0 & 0 & A_{k+1,k+1} &  &  \\ 
\vdots  &   & \vdots  & \vdots  & \ddots  & \ddots  &  \\ 
0 & \cdots  & 0 & 0 & \cdots  & 0 & A_{k+r,k+r}%
\end{bmatrix}%
\end{equation*}%
with blocks $A_{ii}$ irreducible of order $n_{i},$ or zero of size $1\times
1,$ such that $\sum\limits_{i=1}^{k+r}n_{i}=n,$ and $%
\begin{bmatrix}
A_{i1} & A_{i2} & \cdots A_{i,i-1}%
\end{bmatrix}%
$ nonzero, $i=2,\dots ,k$. We may assume, without loss of generality, that $%
\lambda _{1}$ is an eigenvalue of $A_{11}\in \mathcal{CS}_{\lambda _{1}},$
and $A_{ii}\in \mathcal{CS}_{\rho (A_{ii})}$, $i=2,3,\ldots ,k+r$.

From Lemma \ref{lema1}, the submatrix 
\begin{equation*}
A_{1}=%
\begin{bmatrix}
A_{11} & 0 \\ 
A_{21} & A_{22}%
\end{bmatrix}%
,
\end{equation*}%
in the left upper corner of $\widetilde{A},$ is similar to a nonnegative
matrix $B_{1}\in \mathcal{CS}_{\lambda _{1}},$ with $%
B_{1}=D_{1}{}^{-1}A_{1}D_{1}$.

We define $\widetilde{D}_{1}=%
\begin{bmatrix}
D_{1} &  \\ 
& I_{n-(n_{1}+n_{2})}%
\end{bmatrix}%
$. Then%
\begin{equation*}
\widetilde{D}_{1}^{-1}\widetilde{A}\widetilde{D}_{1}=%
\begin{bmatrix}
B_{1} &  &  &  &  &  &  \\ 
\ast  & A_{33} &  &  &  &  &  \\ 
\vdots  & \ddots  & \ddots  &  &  &  &  \\ 
\ast  & \cdots  & \ast  & A_{kk} &  &  &  \\ 
0 & \cdots  & 0 & 0 & A_{k+1,k+1} &  &  \\ 
\vdots  &  & \vdots  & \vdots  & \ddots  & \ddots  &  \\ 
0 & \cdots  & 0 & 0 & \cdots  & 0 & A_{k+r,k+r}%
\end{bmatrix}%
.
\end{equation*}%
Again, from Lemma \ref{lema1}, the left upper corner submatrix of $%
\widetilde{D}_{1}^{-1}\widetilde{A}\widetilde{D}_{1},$%
\begin{equation*}
A_{2}=%
\begin{bmatrix}
B_{1} & 0 \\ 
\ast  & A_{33}%
\end{bmatrix}%
,
\end{equation*}%
is similar to a nonnegative matrix $B_{2}\in \mathcal{CS}_{\lambda _{1}},$
with $B_{2}=D_{2}{}^{-1}A_{2}D_{2}.$ Then we define $\widetilde{D}_{2}=%
\begin{bmatrix}
D_{2} &  \\ 
& I_{n-(n_{1}+n_{2}+n_{3})}%
\end{bmatrix}%
$ and we obtain
\begin{equation*}
\widetilde{D}_{2}^{-1}\widetilde{D}_{1}^{-1}\widetilde{A}\widetilde{D}_{1}%
\widetilde{D}_{2}=%
\begin{bmatrix}
B_{2} &  &  &  &  &  &  \\ 
\ast  & A_{44} &  &  &  &  &  \\ 
\vdots  & \ddots  & \ddots  &  &  &  &  \\ 
\ast  & \cdots  & \ast  & A_{kk} &  &  &  \\ 
0 & \cdots  & 0 & 0 & A_{k+1,k+1} &  &  \\ 
\vdots  &   & \vdots  & \vdots  & \ddots  & \ddots  &  \\ 
0 & \cdots  & 0 & 0 & \cdots  & 0 & A_{k+r,k+r}%
\end{bmatrix}%
.
\end{equation*}%
Proceeding in a similar way, after $k-1$ steps, we obtain 
\begin{equation*}
\widetilde{D}_{k-1}^{-1}\cdots \widetilde{D}_{1}^{-1}\widetilde{A}\widetilde{%
D}_{1}\cdots \widetilde{D}_{k-1}=%
\begin{bmatrix}
B_{k-1} &  &  &  \\ 
& A_{k+1,k+1} &  &  \\ 
&  & \ddots  &  \\ 
&  &  & A_{k+r,k+r}%
\end{bmatrix}%
,
\end{equation*}%
which is a block diagonal matrix, with $B_{k-1}\in \mathcal{CS}_{\lambda
_{1}}$. Now, from Lemma \ref{lema2}, the submatrix 
\begin{equation*}
A_{k}^{\prime }=%
\begin{bmatrix}
B_{k-1} &  \\ 
& A_{k+1,k+1}%
\end{bmatrix}%
,
\end{equation*}%
is similar to a nonnegative matrix $B_{k}^{\prime }\in \mathcal{CS}_{\lambda
_{1}},$ $B_{k}^{\prime }=D_{k}{}^{-1}S_{k}{}^{-1}A_{k}^{\prime }S_{k}D_{k},$
where $S_{k}=%
\begin{bmatrix}
I_{n_{1}+\cdots +n_{k}} &  \\ 
-\mathbf{e}\mathbf{z}_{k}^{T} & I_{k+1,k+1}%
\end{bmatrix}%
,$ with $\mathbf{z}_{k}$ being an eigenvector of $B_{k-1}^{T}$ associated to $%
\lambda _{1}.$\newline
We define $\widetilde{D}_{k}=%
\begin{bmatrix}
S_{k}D_{k} &  \\ 
& I_{n-(n_{1}+\cdots +n_{k+1})}%
\end{bmatrix}%
.$ Then, 
\begin{equation*}
\widetilde{D}_{k}^{-1}\cdots \widetilde{D}_{1}^{-1}\widetilde{A}\widetilde{D}%
_{1}\cdots \widetilde{D}_{k}=%
\begin{bmatrix}
B_{k}^{^{\prime }} &  &  &  \\ 
& A_{k+2,k+2} &  &  \\ 
&  & \ddots  &  \\ 
&  &  & A_{k+r,k+r}%
\end{bmatrix}%
.
\end{equation*}%
Proceeding in a similar way, after $r-1$ steps, we obtain a nonnegative
matrix $B\in \mathcal{CS}_{\lambda _{1}}$ similar to $A$.
\end{proof}

\begin{remark}
Note that the condition of simple Perron eigenvalue cannot be deleted from
Theorem \ref{Teorema}, as shown in matrix%
\begin{equation*}
\left[ 
\begin{array}{cc}
1 & 0 \\ 
1 & 1%
\end{array}%
\right] .
\end{equation*}%
Observe also that this means that it is not always possible to work with
matrices with constant row sums in the NIEDP, this fact does not apply to the NIEP.
\end{remark}

\section{Perturbation of universally realizable lists}

Guo in 1997 \cite{Guo} proved that increasing the Perron eigenvalue of a realizable list preserves the realizability. 
We extend this result to $\mathcal{UR}$ lists.

\begin{theorem}
Let $\Lambda =\{\lambda _{1},\lambda _{2},\ldots ,\lambda _{n}\}$ be a list
of complex numbers with $\lambda _{1}$ simple. If $\Lambda $ is $\mathcal{UR}$,  then 
$\Lambda _{\epsilon }=\{\lambda _{1}+\epsilon ,\lambda _{2},\ldots ,\lambda
_{n}\}$ is also $\mathcal{UR}$ for any $\epsilon >0$.
\end{theorem}

\begin{proof}
Let $\epsilon >0$ and 
$$J_{\epsilon }=J_1(\lambda _{1}+\epsilon )\bigoplus\limits_{i=2}^{k}J_{n_{i}}(\lambda _{i})$$
 be a $JCF$ allowed by $\Lambda _{\epsilon }$. The matrix 
$$J=J_1(\lambda _{1})\bigoplus\limits_{i=2}^{k}J_{n_{i}}(\lambda _{i})$$
is an allowed JCF by $\Lambda $. Because $\Lambda $ is $\mathcal{UR}$, 
there exists a nonnegative matrix $A$ with spectrum $\Lambda $ and Jordan canonical form $J$.
 Besides, from Theorem \ref{Teorema}, there exists a nonnegative matrix $B\in \mathcal{CS}%
_{\lambda _{1}}$ with $J(B)=J$. Then, from Theorem \ref{Soto Ccapa}, for $B$ and $%
\mathbf{q}^{T}=[\frac{\epsilon }{n},\ldots ,\frac{\epsilon }{n}],$  we have that the matrix $%
A_{\epsilon }=B+\mathbf{eq}^{T}$ is nonnegative with spectrum $\Lambda_{\epsilon }$ and $JCF$ 
\begin{equation*}
J(A_{\epsilon })=J(B)+\epsilon E_{11}=J+\epsilon E_{11}=J_1(\lambda
_{1}+\epsilon )\bigoplus\limits_{i=2}^{k}J_{n_{i}}(\lambda _{i}).
\end{equation*}%
Thus, $\Lambda _{\epsilon }$ is $\mathcal{UR}.$
\end{proof}\\

Guo in 1997 \cite{Guo} also proved that increasing by $\epsilon $ a Perron eigenvalue  and decreasing by $\epsilon $
another real eigenvalue of a realizable list preserves the realizability.
Soto and Ccapa in 2008 \cite{Soto Ccapa} proved that a list of real numbers of Sule\v{\i}manova
type, that is, a list $\{\lambda _{1},\lambda _{2},\ldots ,\lambda _{n}\}$ with
$\lambda _{i}\leq 0$ for $i=2,\ldots ,n,$ and $\sum_{i=1}^n\lambda _{1}\geq 0 $,
 is $\mathcal{UR}.$ As a consequence, 
 the perturbed list $\{\lambda _{1}+\epsilon ,\lambda _{2}-
\epsilon ,\lambda _{3},\ldots ,\lambda _{n}\}$ with $\epsilon >0$ is $\mathcal{UR}$
for nonnegative lists $\{\lambda _{1},\lambda_{2},\ldots ,\lambda _{n}\}$ and also for Sule\v{\i}manova type lists $\{\lambda_{1},\lambda_{2},\ldots ,\lambda _{n}\}$.
 As we show below, this is not true for general lists $\{\lambda _{1},\lambda_{2},\ldots ,\lambda _{n}\}$. 
 The construction of a counter-example is based on the study of $\mathcal{UR}$ lists of size 5 with trace zero and three negative elements. 
This construction has been motivated by the work of Cronin and Laffey \cite{Cronin}. They show
that a realizable list
is not necessarily diagonalizably realizable.
In particular, they observe that the lists $\{3+t, 3-t, -2+\epsilon , -2, -2-\epsilon \}$ 
are realizable for small positive values of $\epsilon $ and values of $t$ close to 0.44, 
but they are symmetrically realizable only for $t\geq 1-\epsilon $ {\cite[Theorem 3] {Spector}}.
Note that these lists are  diagonalizably realizable, since the eigenvalues are distinct.
 However, this is not a continuous property in $\epsilon $ 
as Cronin and Laffey   show via the following result.

\begin{proposition}  {\rm \cite{Cronin}} \label{Cronin}
Suppose $\{3+t,3-t,-2,-2,-2\}$ is diagonalizably realizable, then $t\geq 1$.
\end{proposition}

Note that the list $\{3+t,3-t,-2,-2,-2\}$ represents any list of size $5$ with trace zero, simple Perron eigenvalue and three negative elements all equal, {\it i.e.}, lists of the form
$\{\lambda _{1},\lambda _{2},\lambda _{3},\lambda _{3},\lambda _{3}\}$ with $\lambda _{1}>\lambda _{2}\geq 0> \lambda _{3}$  and $\lambda _{1}+\lambda _{2}+3\lambda _{3}=0$.
This list can be scaled by $-2/\lambda _3$ to 
\begin{equation*}
\left\{\frac{-2\lambda _1}{\lambda _3}, \frac{-2\lambda _2}{\lambda _3},-2,-2,-2\right\}
\end{equation*}%
and taking $t=\frac{-2\lambda _1}{\lambda _3}-3=3+\frac{2\lambda _2}{\lambda _3}$ we have 
\begin{equation*}
\Lambda _{\pm t}=\{3+t,3-t,-2,-2,-2\},\text{ \ }0<t\leq 3.
\end{equation*}%

Analogously:
\begin{itemize}
\item  The list $\Lambda  _t^{t_0}=\{3+t-t_0, 3-t, -2+t_0,-2,-2\}$, with $0<t_0<\min\{1+t,2t\}<2$ and
$0< t \leq 3$,
represents the lists $\{\lambda _1, \lambda _2, \lambda _3, \lambda _4, \lambda _4\}$ with $\lambda _{1}>\lambda _{2}\geq 0 > \lambda _{3}> \lambda _4$  and 
$\lambda _{1}+\lambda _{2}+\lambda _{3}+2\lambda _4=0$ (scaling by $-2/\lambda _4$ and taking $t_0=2-\frac{2\lambda _3}{\lambda _4}$  and  
$t=\frac{-2\lambda _1}{\lambda _4}-3+t_0=3+\frac{2\lambda _2}{\lambda _4}$).
\item
The list $\Lambda  _t^{'t_0}=\{3+t+t_0, 3-t, -2,-2,-2-t_0\}$, 
with $t_0> \max\{0, -2t\}$ and  $-1<t\leq 3$, represents the lists
$\{\lambda _1, \lambda _2, \lambda _3, \lambda _3, \lambda _4\}$ with $\lambda _{1}>\lambda _{2}\geq 0 > \lambda _{3}> \lambda _4>-\lambda _{1}$ and 
$\lambda _{1}+\lambda _{2}+2\lambda _{3}+\lambda _4=0$ (scaling by $-2/\lambda _3$ and taking $t_0=-2+\frac{2\lambda _4}{\lambda _3}$  and  
$t=\frac{-2\lambda _1}{\lambda _3}-3-t_0=3+\frac{2\lambda _2}{\lambda _3}$).
\end{itemize}

We need the following result due to \v{S}migoc:

\begin{lemma} {\rm \cite[Lemma 5]{Smigoc}} \label{Smigoc}
 Suppose $B$ is an $m \times m$ matrix with Jordan canonical form $J(B)$ that contains at least one $1 \times 1$ Jordan block corresponding
 to the eigenvalue c:
 
$J(B)= \begin{bmatrix}
c & 0 \\ 
0  & I(B)%
\end{bmatrix}.
$

\noindent Let ${\bf u}$ and ${\bf v}$, respectively, be  left and right eigenvectors of $B$ associated with the $1 \times 1$ Jordan block in the above canonical form.
Furthermore, we normalize vectors ${\bf u}$ and ${\bf v}$ so that ${\bf u}^T{\bf v}=1$. Let $J(A)$ be a Jordan canonical form for an $n \times n$ matrix

$A= \begin{bmatrix}
A_1 & {\bf a} \\ 
{\bf b}^T  & c%
\end{bmatrix},
$

\noindent where $A_1$ is an $(n-1) \times (n-1)$ matrix and ${\bf a}$ and ${\bf b}$ are vectors in $\mathbb{C}^{n-1}$. Then the matrix

$C= \begin{bmatrix}
A_1 & {\bf a}{\bf u}^T \\ 
{\bf v}{\bf b}^T  & B%
\end{bmatrix}
$

\noindent has Jordan canonical form

$J(C)= \begin{bmatrix}
J(A) & 0 \\ 
0  & I(B)%
\end{bmatrix}.
$
\end{lemma}

We consider the lists $\Lambda  _t^{t_0}=\{3+t-t_0, 3-t, -2+t_0,-2,-2\}$ and $\Lambda  _t^{'t_0}=\{3+t+t_0, 3-t, -2,-2,-2-t_0\}$ 
that have a better behavior than $\Lambda _{\pm t}$ with respect to the Guo result applied to $\mathcal{UR}$.

\begin{theorem} \label{CM}
\noindent i) Let $\Lambda  _t^{t_0}=\{3+t-t_0, 3-t, -2+t_0,-2,-2\}$ with $0<t_0<2$ and  $\, \frac{t_0}{2}<t\leq 3$. \, 
If   $\, \Lambda  _t^{t_0}$ is realizable, then it is  $\mathcal{UR}$.

\noindent ii) Let $\Lambda  _t^{'t_0}=\{3+t+t_0, 3-t, -2,-2,-2-t_0\}$ with $t_0>\max\{0,-2t\}\,$ and  $\,t\leq 3$. \,
 If   $\, \Lambda  _t^{'t_0}$ is realizable, then it is  $\mathcal{UR}$.
\end{theorem} 

\begin{proof}
{\it i)} Observe that the list $\Lambda  _t^{t_0}$ has two possible JCF, since the only repeated eigenvalue is $-2$ with double multiplicity.

Under the realizability conditions in \cite{Laffey, Torre}, the realizing matrices
for lists $\Lambda _{t}^{t_{0}}$ have the form 
$$
A=\left[ \begin{array} {ccccc}
0&1&0&0&0 \\
*&0&1&0&0 \\
*&*&0&1&0 \\
*&*&*&0&1 \\
*&*&*&*&0 
\end{array}
\right],
$$
then $\mbox{rank}(A+2I)=4$ and $A$ has a $JCF$ with a Jordan block of size two $J_{2}(-2)$. 

If $\Lambda  _t^{t_0}$ is symmetrically realizable (see Spector conditions in \cite[Theorem $3$]{Spector}), then $\Lambda  _t^{t_0}$ is $\mathcal{DR}$.

If $\Lambda  _t^{t_0}$ is realizable but not symmetrically realizable, which means that $t<1$ (see next section), we show that $\Lambda  _t^{t_0}$ is $\mathcal{DR}$
via the \v{S}migoc method given in Lemma \ref{Smigoc}.
Let 
\begin{equation*}
\Gamma   _{1}=\{3+t-t_{0},3-t,-2+t_{0},-2\} \; \text{and} \; \Gamma 
_{2}=\{\mbox{tr}(\Gamma  _{1}),-2\}=\{2,-2\}.
\end{equation*}
Note that these spectra are realizable because they satisfy the Perron and trace conditions. 
The matrix 
$$B=\left[
\begin{array}{cc}
  0 & 2    \\
  2 & 0   \\
\end{array}
\right]\approx J(B)=\left[
\begin{array}{cc}
  c=2 & 0    \\
  0 & -2   \\
\end{array}
\right]
$$ 
realizes  $\Gamma_2$. Let ${\bf u}^{T}=\left[ 1/2,1/2\right] $ and ${\bf v}^{T}=\left[ 1,1\right]$ be, respectively, left and right normalized eigenvectors  of $B$.

We need to find a realization of $\Gamma  _{1}$ with diagonal $(0,0,0,c=\mbox{tr}(\Gamma  _{1})=2)$ and the only realization that we know with this diagonal
is the one given in [16, Theorem 14] which is of the form
\begin{equation*}
A=\left[ 
\begin{array}{ccc|c}
0 & 1 & 0 & 0 \\ 
d_{1} & 0 & 1 & 0 \\ 
b & 0 & 0 & 1 \\ 
\hline
a & 0 & d_{3} & 2%
\end{array}%
\right]. 
\end{equation*}%
The characteristic polynomial of $A$ is 
$$\begin{array}{lll}
P_A(x)   & = & x^4 -2x^3-(d_1+d_3)x^2+(2d_1-b)x+2b+d_1d_3-a  \\
         & = & (x-(3+t-t_0))(x-(3-t))(x-(-2+t_0))(x+2) \\
         & = & x^4+k_1x^3+k_2x^2+k_3x+k_4 
\end{array}
$$
with  
\begin{equation*} 
\begin{array}{lll} 
k_2  & = & -(t^2 - t_0t + t_0^2 - 5t_0 + 11), \\
k_3  & = & (t_0 - 4)t^2 + t_0(4 - t_0)t + t_0^2 - 5t_0 + 12, \\ 
k_4  & = & 2(t_0 - 2)(t - t_0 + 3)(t - 3).
\end{array}
\end{equation*}
Identifying coefficients we have the system:
\begin{equation} \label{sistema}
d_1+d_3=-k_2, \quad 2d_1-b=k_3, \quad 2b+d_1d_3-a=k_4 
\end{equation}
which allows us to obtain realizations of $\Gamma _1$, in function of  $d_1$, of the form
$$A(d_1)=\left[
\begin{array}{ccc|c}
   0 & 1 & 0   &  0    \\
  d_1 & 0 & 1   &  0   \\
 2d_1-k_3  & 0 & 0   &  1    \\
 \hline 
 -d_1^2+(4-k_2)d_1-2k_3-k_4  & 0 & -k_2-d_1 &  2  
\end{array}
\right]
$$
that has JCF
$$
J(A(d_1))=\left[
\begin{array}{cccc}
   3+t-t_0 & 0 & 0   &  0    \\
   0 & 3-t   &  0  & 0 \\
 0  & 0 & -2+t_0   &  0    \\
 0  & 0 & 0 &  -2  
\end{array}
\right].
$$ 
Now, by Lemma \ref{Smigoc}, the bonding of matrices $A(d_{1})$ and $B$ leads to the matrix
\begin{equation*} 
C(d_1)=\left[
\begin{array}{ccc|cc}
  0   &  1  &  0   &  0  &  0  \\
 d_1   &  0  &  1   &  0  &  0  \\
2d_1-k_3   &  0  &  0   & 1/2 & 1/2    \\
\hline
 -d_1^2+(4-k_2)d_1-2k_3-k_4  &  0  & -k_2-d_1  &  0  &  2    \\
 -d_1^2+(4-k_2)d_1-2k_3-k_4  &  0  & -k_2-d_1  &  2  &  0   
\end{array}
\right]
\end{equation*} 
which realizes diagonally the list $\Lambda _{t}^{t_{0}}$.

Finally, $\Lambda _{t}^{t_{0}}$ is $\mathcal{UR}$.

\noindent {\it ii)} Analogously, under the realizability conditions in \cite{Laffey, Torre}, the realizing matrices
for lists $\Lambda _{t}^{'t_{0}}$ have a $JCF$ with a Jordan block of size two $J_{2}(-2)$. 

If $\Lambda  _t^{'t_0}$ is symmetrically realizable, then $\Lambda  _t^{'t_0}$ is $\mathcal{DR}$.

If $\Lambda  _t^{'t_0}$ is realizable but not symmetrically realizable (for $t<1$), we apply the \v{S}migoc method to the spectra
$$
\Gamma  _1^{'}=\{3+t+t_0, 3-t, -2,-2-t_0\} \; \mbox{and} \;   \Gamma  _2=\{2, -2\}
$$
and, in the same way, we obtain  the following $\mathcal{DR}$ realization of $\Lambda  _t^{'t_0}$
\begin{equation*} 
C(d_1)=\left[
\begin{array}{ccccc}
  0   &  1  &  0   &  0  &  0  \\
 d_1   &  0  &  1   &  0  &  0  \\
2d_1-k_3   &  0  &  0   & 1/2 & 1/2    \\
 -d_1^2+(4-k_2)d_1-2k_3-k_4  &  0  & -k_2-d_1  &  0  &  2    \\
 -d_1^2+(4-k_2)d_1-2k_3-k_4  &  0  & -k_2-d_1  &  2  &  0   
\end{array}
\right]
\end{equation*}
for the system (\ref{sistema}), with 
\begin{equation*} 
\begin{array}{lll} 
k_2  & = & -(t^2 - t_0t + t_0^2 + 5t_0 + 11), \\
k_3  & = & (t_0 + 4)t^2 + t_0(4 + t_0)t - t_0^2 - 5t_0 - 12, \\ 
k_4  & = & 2(t_0 + 2)(t + t_0 + 3)(3-t).
\end{array}
\end{equation*}

Hence, $\Lambda _{t}^{'t_{0}}$ is $\mathcal{UR}$.
\end{proof}

\begin{corollary}
Let $\Lambda =\{\lambda _{1},\lambda _{2},\ldots ,\lambda _{n}\}$ be a $\mathcal{UR}$ list with $\lambda _2$ real.
 The list $\{\lambda _{1}+\epsilon ,\lambda _{2}-\epsilon ,\lambda _{3},\ldots,\lambda _{n}\}$, for $\epsilon >0$, is not necessarily $\mathcal{UR}$.
\end{corollary}

\begin{proof}
Let 
\begin{equation*}
\Lambda =\Lambda   _{t}^{t_{0}} = \{3+t-t_0,3-t,-2+t_{0},-2,-2\}
\end{equation*}%
be a $\mathcal{UR}$ list as in Theorem \ref{CM} with $t<1$ (see Lemma \ref{regions}  for its existence).
Now, applying Wuwen perturbation with $\epsilon =t_0$, we obtain the list 
$$
\{3+t,3-t,-2,-2,-2\}
$$
which is not diagonalizably realizable by Proposition \ref{Cronin} and therefore it is not $\mathcal{UR}$.
\end{proof}\\

It is easy to see that if $\Lambda $ and $\Gamma $ are
lists of nonnegative real numbers, then $\Lambda \cup \Gamma $ is $%
\mathcal{UR}.$ Let%
\begin{equation*}
\Lambda =\{\lambda _{1},\lambda _{2},\ldots ,\lambda _{n}\} \; \text{ and } \; %
\Gamma =\{\mu _{1},\mu _{2},\ldots ,\mu _{m}\}
\end{equation*}%
be lists of real numbers of Sule\v{\i}manova type with trace zero and $\lambda _{1}>\mu _{1}$,  the Perron eigenvalues of $\Lambda $ and $\Gamma $ respectively. 
Then, from \cite{Collao}, $\Lambda \cup \Gamma $ is $\mathcal{UR}$. Now we show that this is not true for general lists.

\begin{lemma}
\label{lemma3} Let $\Lambda=\{\lambda_{1},\lambda_{1},\lambda_{2},%
\lambda_{2}\}$ be a list of real numbers with $\lambda_{1}>0>\lambda_{2}\geq -\lambda_{1}$ 
 and $\lambda_{1}+2\lambda_{2}<0$. Then $\Lambda $ has no nonnegative realization  with Jordan canonical form 
\begin{equation*}
J=%
\begin{bmatrix}
\lambda_{1} & 0 & 0 & 0 \\ 
0 & \lambda_{1} & 0 & 0 \\ 
0 & 0 & \lambda_{2} & 1 \\ 
0 & 0 & 0 & \lambda_{2}%
\end{bmatrix}%
.
\end{equation*}
\end{lemma}

\begin{proof}
Suppose there exists a nonnegative realization $A$ of $\Lambda $ with Jordan canonical form $J(A)=J$. As $\lambda_1+2\lambda_2<0$, then $\Lambda$ only admits reducible realizations and must be partitioned as $\{\lambda_1,\lambda_2\}\cup\{\lambda_1,\lambda_2\}$. So we assume, without loss of
generality, that $A$ is of the form 
\begin{equation*}
A=%
\begin{bmatrix}
B & 0 \\ 
C & D%
\end{bmatrix}%
,
\end{equation*}%
where $B$ and $D$ are irreducible matrices with spectrum $\{\lambda
_{1},\lambda _{2}\}$. Therefore, from the minimal polynomial of $B$ and $D,$
we have%
\begin{equation*}
B^{2}=(\lambda _{1}+\lambda _{2})B-\lambda _{1}\lambda _{2}I\ \ \text{and}\
\ D^{2}=(\lambda _{1}+\lambda _{2})D-\lambda _{1}\lambda _{2}I.
\end{equation*}%
Since the minimal polynomial of $A$ is 
\begin{equation*}
x^{3}+(-\lambda _{1}-2\lambda _{2})x^{2}+(2\lambda _{1}\lambda _{2}+\lambda
_{2}^{2})x-\lambda _{1}\lambda _{2}^{2},
\end{equation*}%
then 
\begin{equation*}
A^{3}+(-\lambda _{1}-2\lambda _{2})A^{2}+(2\lambda _{1}\lambda _{2}+\lambda
_{2}^{2})A-\lambda _{1}\lambda _{2}^{2}I=0,
\end{equation*}%
with%
\begin{equation*}
A^{2}=%
\begin{bmatrix}
B^{2} & 0 \\ 
CB+DC & D^{2}%
\end{bmatrix}%
=%
\begin{bmatrix}
(\lambda _{1}+\lambda _{2})B-\lambda _{1}\lambda _{2}I & 0 \\ 
CB+DC & (\lambda _{1}+\lambda _{2})D-\lambda _{1}\lambda _{2}I%
\end{bmatrix}%
,
\end{equation*}%
and 
\begin{align*}
A^{3}& =AA^{2}=%
\begin{bmatrix}
(\lambda _{1}+\lambda _{2})B^{2}-\lambda _{1}\lambda _{2}B & 0 \\ 
(\lambda _{1}+\lambda _{2})CB-\lambda _{1}\lambda _{2}C+DCB+D^{2}C & 
(\lambda _{1}+\lambda _{2})D^{2}-\lambda _{1}\lambda _{2}D%
\end{bmatrix}
\\
& \\
& =%
\begin{bmatrix}
(\lambda _{1}^{2}+\lambda _{1}\lambda _{2}+\lambda _{2}^{2})B-(\lambda
_{1}^{2}\lambda _{2}+\lambda _{1}\lambda _{2}^{2})I & 0 \\ 
(\lambda _{1}+\lambda _{2})(CB+DC)+DCB-2\lambda _{1}\lambda _{2}C & (\lambda
_{1}^{2}+\lambda _{1}\lambda _{2}+\lambda _{2}^{2})D-(\lambda
_{1}^{2}\lambda _{2}+\lambda _{1}\lambda _{2}^{2})I%
\end{bmatrix}%
.
\end{align*}%
Therefore, \newline
$A^{3}-(\lambda _{1}+2\lambda _{2})A^{2}+(2\lambda _{1}\lambda _{2}+\lambda
_{2}^{2})A-\lambda _{1}\lambda _{2}^{2}I$ 
\begin{equation*}
=%
\begin{bmatrix}
(\lambda _{1}^{2}+\lambda _{1}\lambda _{2}+\lambda _{2}^{2})B-(\lambda
_{1}^{2}\lambda _{2}+\lambda _{1}\lambda _{2}^{2})I & 0 \\ 
(\lambda _{1}+\lambda _{2})(CB+DC)+DCB-2\lambda _{1}\lambda _{2}C & (\lambda
_{1}^{2}+\lambda _{1}\lambda _{2}+\lambda _{2}^{2})D-(\lambda
_{1}^{2}\lambda _{2}+\lambda _{1}\lambda _{2}^{2})I%
\end{bmatrix}%
\end{equation*}%
\begin{equation*}
-(\lambda _{1}+2\lambda _{2})%
\begin{bmatrix}
(\lambda _{1}+\lambda _{2})B-\lambda _{1}\lambda _{2}I & 0 \\ 
CB+DC & (\lambda _{1}+\lambda _{2})D-\lambda _{1}\lambda _{2}I%
\end{bmatrix}%
+(2\lambda _{1}\lambda _{2}+\lambda _{2}^{2})%
\begin{bmatrix}
B & 0 \\ 
C & D%
\end{bmatrix}%
\end{equation*}%
$-\lambda _{1}\lambda _{2}^{2}I=0$.\newline
Now, by equalizing the block in position ($2,1$) to zero, we have:%
\begin{eqnarray*}
(\lambda _{1}+\lambda _{2})(CB+DC)+DCB-2\lambda _{1}\lambda _{2}C -(\lambda _{1}+2\lambda _{2})(CB+DC)+(2\lambda _{1}\lambda _{2}+\lambda
_{2}^{2})C
\end{eqnarray*}
\begin{eqnarray*}
=-\lambda _{2}(CB+DC)+DCB+\lambda _{2}^{2}C=0.
\end{eqnarray*}%
Since the matrices involved in the last equality are nonnegative and $\lambda _{2}<0$, 
this is only possible if each addend is zero. In particular, $C=0$. Then 
\begin{align*}
\dim(\ker(A-\lambda _{2}I))& =4-\mbox{rank}(A-\lambda _{2}I) \\
& =4-\mbox{rank}%
\begin{bmatrix}
B-\lambda _{2}I & 0 \\ 
0 & D-\lambda _{2}I%
\end{bmatrix}
\\
& =4-(\mbox{rank}(B-\lambda _{2}I)+\mbox{rank}(D-\lambda _{2}I)) \\
& =4-(1+1)=2.
\end{align*}
However, from  $J(A)=J$ we have 
\begin{equation*}
\mbox{dim}(\mbox{ker}(A-\lambda _{2}I))=4-\mbox{rank}%
\begin{bmatrix}
\lambda _{1}-\lambda _{2} & 0 & 0 & 0 \\ 
0 & \lambda _{1}-\lambda _{2} & 0 & 0 \\ 
0 & 0 & 0 & 1 \\ 
0 & 0 & 0 & 0%
\end{bmatrix}%
=1,
\end{equation*}%
which contradicts the existence of a nonnegative realization $A$ with Jordan canonical form $J$.
\end{proof}\\

As an example, consider $\Lambda =\{1,-1\}.$ It
is clear that $\Lambda $ is $\mathcal{UR}$. However, from
Lemma \ref{lemma3}, the list $\Lambda \cup \Lambda = \{1,1,-1,-1\}$ has no
nonnegative realization with $JCF$ 
\begin{equation*}
J=%
\begin{bmatrix}
1 & 0 & 0 & 0 \\ 
0 & 1 & 0 & 0 \\ 
0 & 0 & -1 & 1 \\ 
0 & 0 & 0 & -1%
\end{bmatrix}%
.
\end{equation*}%
Therefore, $\Lambda \cup \Lambda $ is not $\mathcal{UR}$.

\section{Lists of size 5 with trace zero and three negative elements}

We are interested in the realizability of the lists with size 5 and trace zero 
\begin{equation*}
\Lambda _{\pm t}=\{3+t,3-t,-2,-2,-2\},
\end{equation*}
\begin{equation*}
\Lambda  _t^{t_0}=\{3+t-t_0, 3-t, -2+t_0,-2,-2\},
\end{equation*}
\begin{equation*}
\Lambda  _t^{'t_0}=\{3+t+t_0, 3-t, -2,-2,-2-t_0\}
\end{equation*}
introduced in Section 3.
It is well known that the list $\Lambda _{\pm t}$ is realizable if and only if $t\geq \sqrt{16\sqrt{6}-39} = 0.43799\cdots $ 
(see \cite{LaMe}), and symmetrically realizable if and only if $t\geq 1$ (see \cite{Spector}). Now, we study when the lists $\Lambda  _t^{t_0}$ 
and $\Lambda  _t^{'t_0}$  are realizable but not symmetrically realizable.
We need  the following result:

\begin{theorem} {\rm \cite[Theorem 39 for $n=5$ and $p=2$] {Torre}} \label{Torre}
 Let $P(x)=x^5+k_2x^3+k_3x^2+k_4x+k_5$. Then the following statements are equivalent: 
\vspace*{-.2cm}
\begin{itemize}
\item [{\it i)}] $P(x)$ is the characteristic polynomial of a nonnegative matrix; 
\vspace*{-.2cm}
\item [{\it ii)}] the coefficients of $P(x)$ satisfy:
\vspace*{-.3cm}
\begin{itemize}
\item [{\it a)}] $k_2, k_3\leq 0$;
\item [{\it b)}] $k_4\leq\frac{k_2^2}{4}$;
\item [{\it c)}] $k_5\leq\left\{
\begin{array}{ll}
	k_2k_3&\quad \mbox{if }\;\;k_4\leq 0,\\
	\vspace*{-.3cm} & \\
	k_3\Big(\frac{k_2}{2}-\sqrt{\frac{k_2^2}{4}-k_4}\Big)&\quad \mbox{if }\;\; k_4>0.\\
\end{array}\right.$
\end{itemize}
\end{itemize}
\end{theorem}

\begin{lemma} \label{regions}
\begin{enumerate}
\item $\Lambda  _t^{t_0}=\{3+t-t_0, 3-t, -2+t_0,-2,-2\}$ with $0< t_0<2t<2$ is realizable, but not symmetrically realizable, in the region
\begin{equation} \label{regionT}
t\geq \frac{t_0+ \sqrt{16\sqrt{6-t_0}(4-t_0)-3t_0^2+52t_0-156}}{2}.
\end{equation}
\item $\Lambda  _t^{'t_0}=\{3+t+t_0, 3-t, -2,-2,-2-t_0\}$ with $0<t_0,t<1$ and $t+t_0<1$  is realizable, but not symmetrically realizable, in the region
\begin{equation} \label{regionR}
t\geq \frac{-t_0 + \sqrt{16\sqrt{6+t_0}(4+t_0)-3t_0^2-52t_0-156}}{2}.
\end{equation}
\end{enumerate}
\end{lemma}

\begin{proof}
{\it 1.} Note that  $(t_0,t)$ varies in the interior of the triangle $T$ 
with vertices $(0,0), (0,1)$ and $(2,1)$.
The hypothesis $t<1$ guarantees that $\Lambda  _t^{t_0}$ is not symmetrically realizable (see \cite[Theorem $3$]{Spector}). 
 Let us see that $\Lambda  _t^{t_0}$ is realizable using Theorem \ref{Torre}.

The characteristic polynomial $x^5+k_2x^3+k_3x^2+k_4x+k_5$ of $\Lambda   _{t}^{t_{0}}$ is
$$
(x-(3+t-t_{0}))(x-(3-t))(x-(-2+t_{0}))(x+2)^2
$$ 
where 
$$
\begin{array}{l}
k_2=- t^2 +t_0t - t_0^2 + 5t_0 - 15  \\
k_3= -(6-t_0)t^2 + t_0(6 - t_0)t - t_0^2 + 5t_0 - 10\\
k_4 = 4((t_0 - 3)t^2 + t_0(3 - t_0)t + 2t_0^2 - 10t_0 + 15)\\
k_5 = 4(t - 3)(t - t_0 + 3)(t_0 - 2).
\end{array}
$$

Clearly $k_2$ is negative in the triangle $T$ because $k_2<t_0t + 5t_0 - 15 <-3$. 
The derivative of $k_3$ with respect to $t$ is $k_3'= -2(6-t_0)t + t_0(6 - t_0)$, which is $0$ in $t=t_0/2$ and then the maximum value of $k_3$ is 
$k_3(t_0/2)=(2-t_0)(t_0^2-20)/4$ which is negative for $0<t_0<2$ and so $k_3$ is also negative in $T$.

The inequality $k_4\leq \frac{k_2^2}{4}$ holds if and only if $k_2^2-4k_4$ is nonnegative. We have
\begin{equation*}
k_2^2-4k_4=(t^2 - t_0t  + 4(4 - t_0)\sqrt{6 - t_0} + t_0^2 - 13t_0 + 39)(t^2 - t_0t  - 4(4 - t_0)\sqrt{6 - t_0} + t_0^2 - 13t_0 + 39)
\end{equation*}
where the first factor is positive and the second is nonnegative in the triangle $T$ if 
\begin{equation*} 
t\geq \frac{t_0+ \sqrt{16\sqrt{6-t_0}(4-t_0)-3t_0^2+52t_0-156}}{2}.
\end{equation*}
The coefficient $k_4$ is positive in $T$ because 
\begin{equation*}
k_4>4(t_0^3/4 - 3 + (3 - t_0)t_0^2/2 + 2t_0^2 - 10t_0 + 15)=-t_0^3+14t_0^2-40t_0+48>0, 
\end{equation*}
and  $k_5\leq k_3\left(\frac{k_2}{2}-\sqrt{\frac{k_2^2}{4}-k_4}\right)$ in $T$ if the  inequality (\ref{regionT}) holds.

Therefore, by Theorem \ref{Torre}, we conclude that $\Lambda  _t^{t_0}$ is realizable in the region (\ref{regionT}).

\noindent {\it 2.}
Now $(t_0,t)$ varies in the interior of the triangle $R$ with vertices $(0,0), (0,1)$ and $(1,0)$.
Again,  the hypothesis $t<1$ implies no symmetric realization of $\Lambda  _t^{'t_0}$ (see \cite[Theorem $3$]{Spector}). 
\vskip0.3cm
The characteristic polynomial $x^5+k_1x^4+k_2x^3+k_3x^2+k_4x+k_5$ of $\Lambda  _t^{'t_0}$ is
$$
(x-(3+t+t_{0}))(x-(3-t))(x+2)^2(x-(-2-t_{0}))
$$ 
where 
$$
\begin{array}{l}
k_2=-(t^2 + t_0t + t_0^2 + 5t_0 + 15)\\
k_3= -((t_0 + 6)t^2 + t_0(t_0 + 6)t + t_0^2 + 5t_0 + 10)\\
k_4 = - 4((t_0 + 3)t^2 + t_0(t_0 + 3)t - 2t_0^2 - 10t_0 -15) \\
k_5 =  4(3 - t)(t + t_0 + 3)(t_0 + 2).
\end{array}
$$

Clearly $k_2$ and $k_3$ are negative in the triangle $R$. 
For $k_4\leq \frac{k_2^2}{4}$ we have
\begin{equation*}
k_2^2-4k_4=
(t^2 + t_0t  + 4(4 +t_0)\sqrt{6 + t_0} + t_0^2 + 13t_0 + 39)(t^2 + t_0t  - 4(4 + t_0)\sqrt{6 + t_0} +t_0^2 + 13t_0 + 39)
\end{equation*}
where the first factor is positive and the second is nonnegative in the triangle $R$ if 
\begin{equation*} 
t\geq \frac{-t_0 + \sqrt{16\sqrt{6+t_0}(4+t_0)-3t_0^2-52t_0-156}}{2}.
\end{equation*}

The coefficient $k_4$ is positive in $T$ because 
\begin{equation*}
k_4>-4((t_0 + 3)(1-t_0)^2 + t_0(t_0 + 3)(1-t_0) - 2t_0^2 - 10t_0 -15)=12(t_0^2+4t_0+4)>0,
\end{equation*}
and  $k_5\leq k_3\left(\frac{k_2}{2}-\sqrt{\frac{k_2^2}{4}-k_4}\right)$ in $R$ if the inequality (\ref{regionR}) holds.

Therefore, by Theorem \ref{Torre}, we conclude that $\Lambda  _t^{'t_0}$ is realizable in the region (\ref{regionR}).
\end{proof}\\

Figure \ref{Reg1} and Figure \ref{Reg2} show graphically the regions of realizability (the grey regions)
 of $\Lambda_t^{t_0}$ and $\Lambda_t^{'t_0}$ respectively, described in the previous lemma.

\begin{figure}[h]
\mbox{}\hfill
\begin{minipage}[h]{.6\hsize}
\setlength{\unitlength}{1cm}
 \begin{picture}(5,5)
\put(0,0){\epsfxsize=7.2cm \epsfbox{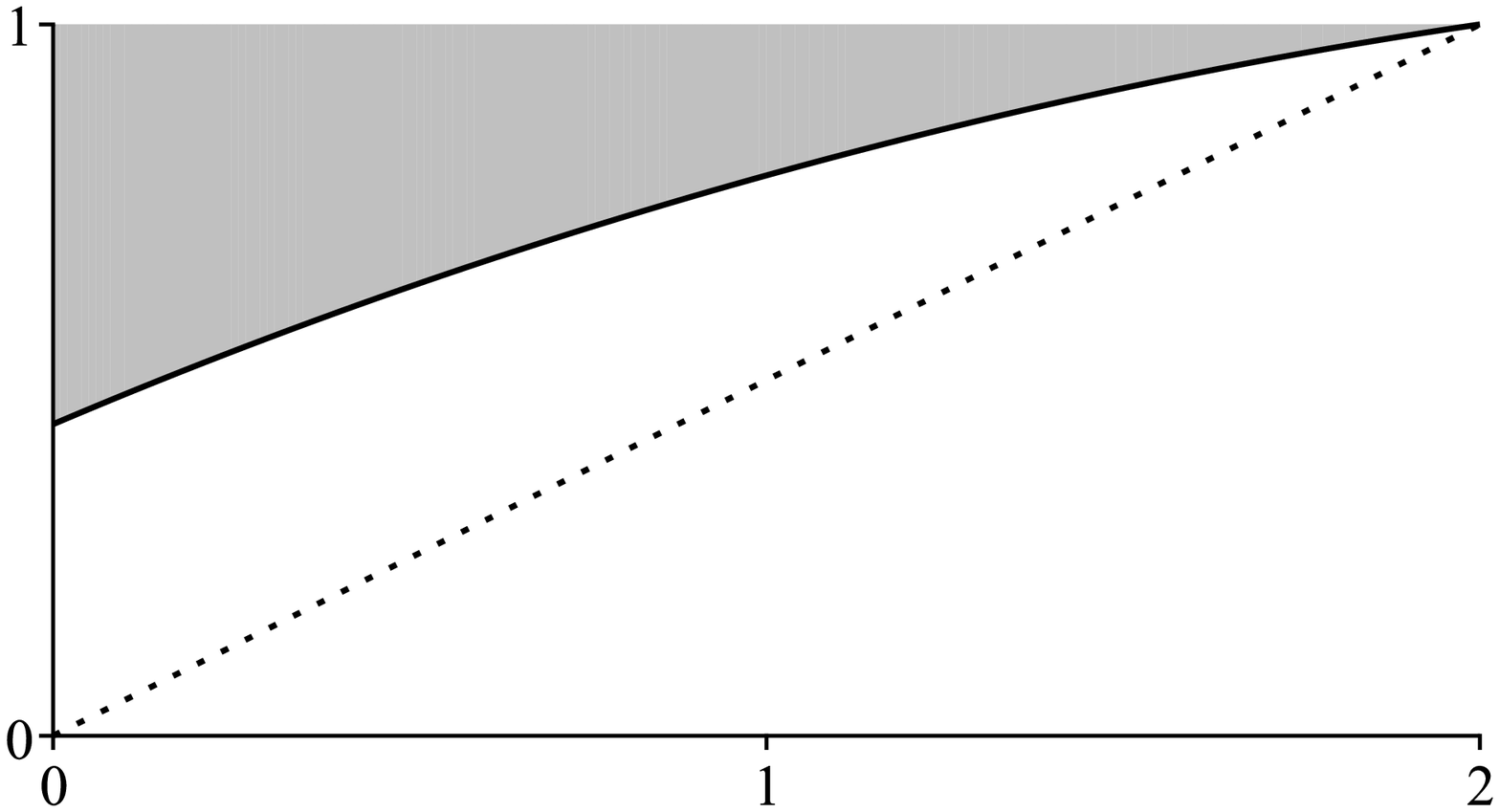}}
\put(6,0){$t_0$}
\put(4.2,2){$t=t_0/2$}
\put(0,2.85){$t$}
\end{picture}
\caption{List  $\Lambda_t^{t_0}.\hspace*{1cm}$}
\label{Reg1}
\end{minipage}
\begin{minipage}[h]{.3\hsize}
\setlength{\unitlength}{1cm}
\begin{picture}(5,5)
\put(0,0){\epsfxsize=4cm \epsfbox{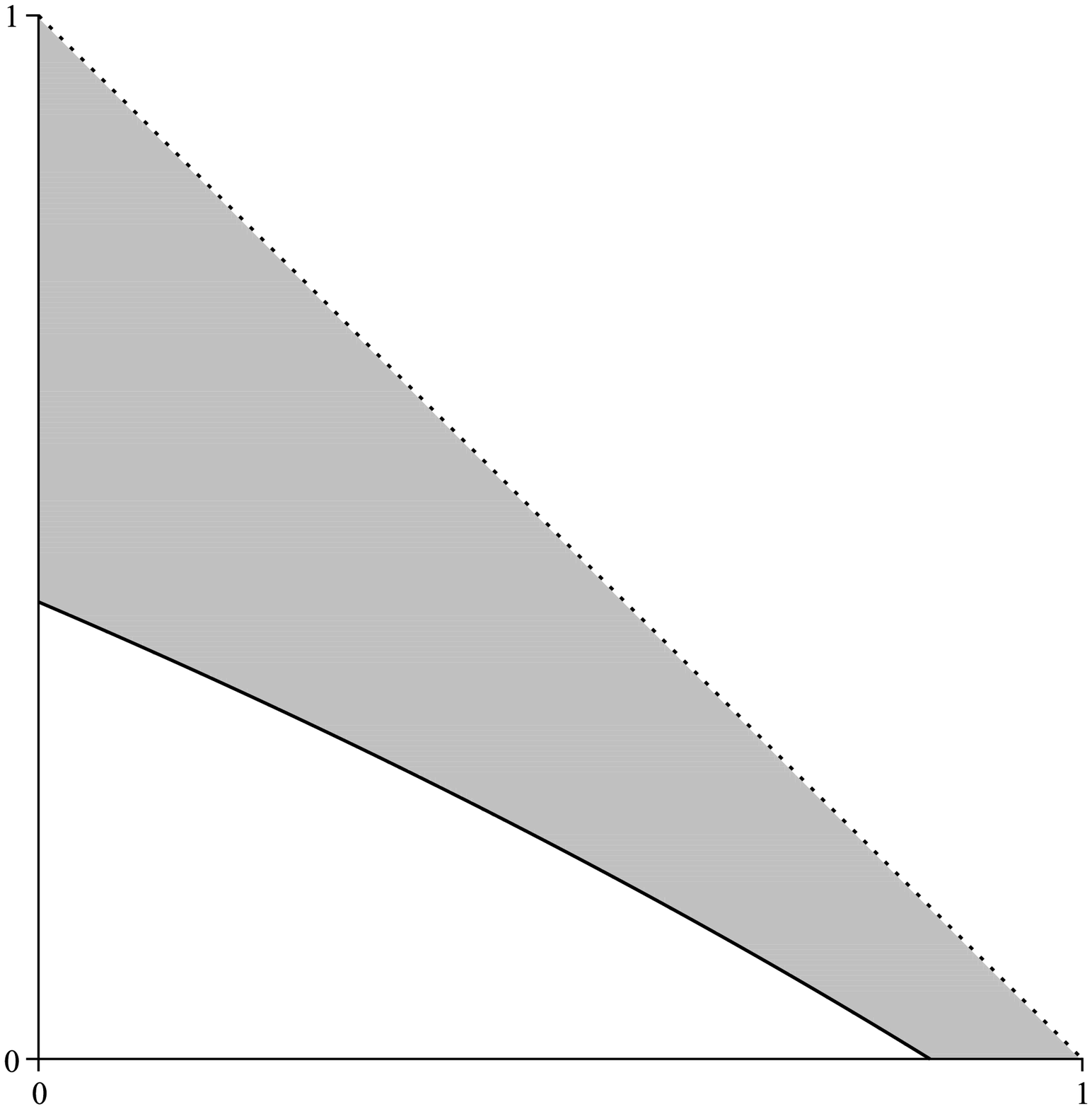}}
\put(3.3,-0.2){$t_0$}
\put(2.3,2){$t_0+t=1$}
\put(-0.1,2.85){$t$}
\end{picture}
\caption{ List $\Lambda_t^{'t_0}.$}
\label{Reg2}
\end{minipage}
\end{figure}

In the following example we give a diagonalizable nonsymmetric realization 
of the lists $\Lambda  _t^{t_0}$ and  $\Lambda  _t^{'t_0}$ for particular values of $(t_0,t)$ in the corresponding regions.

\begin{example} 
Let us consider the list  $\Lambda  _t^{t_0}$  for $t_0=1$.
By Lemma \ref{regions}, the list $\Lambda_t^{1}=\{2+t, 3-t, -1,-2,-2\}$ is realizable for 
$t\geq \frac{1}{2}(1+\sqrt{48\sqrt{5}-107})=   0.7877\cdots $. 
Let us consider $t=0.8$ and realize diagonalizably the list $\Lambda  _{0.8}^{1} =\{2.8, 2.2, -1,-2,-2\}$. 
The characteristic polynomial of the list $\Gamma  _1=\{2.8, 2.2,\break -1,-2\}$ is $x^4 - 2x^3 - \frac{171}{25}x^2 + \frac{212}{25}x + \frac{308}{25}$.
Following the proof of Theorem \ref{CM} we obtain
$$
d_3=\frac{171}{25}-d_1, \quad b=2d_1-\frac{212}{25}, \quad a=-d_1^2+\frac{271}{25}d_1-\frac{732}{25}.
$$
The entries $d_3$ and $b$ are nonnegative for $\frac{106}{25} \leq d_1 \leq \frac{171}{25}$. 
The entry $a$ is nonnegative for 
$d_1\in [\frac{271-\sqrt{241}}{50}, \frac{271+\sqrt{241}}{50}]=[5.10951\cdots , 5.73048\cdots ]$. 
Then the rank of $a$ is between $0$ and its maximum value attained in $d_1=\frac{271}{50}$, i.e., 
$a\in [0,0.094]$.
If we take $d_1=5.5$ we obtain the matrices 

\noindent $A(5.5)=\left[
\begin{array}{cccc}
  0 & 1 & 0   &  0    \\
  5.5 & 0 & 1   &  0   \\
 2.52  & 0 & 0   &  1    \\
 0.09  & 0 & 2.58 &  2  
\end{array}
\right]$ \; and \;
$C(5.5)=\left[
\begin{array}{ccccc}
  0 & 1 & 0   &  0  & 0  \\
  5.5 & 0 & 1   &  0 & 0  \\
 2.52  & 0 & 0   &  0.5 & 0.5   \\
 0.09  & 0 & 2.58 & 0 &  2  \\
 0.09 &  0  & 2.58 &  2  &  0 
\end{array}
\right]$ that realize $\Gamma_1$ and $\Lambda  _{0.8}^{1}$ respectively.

Finally, we consider the list $\Lambda  _t^{'0.5}=\{3.5+t, 3-t, -2,-2,-2.5\}$ that, by Lemma \ref{regions}, is realizable for 
$t\geq \frac{-1+\sqrt{144\sqrt{26}-731}}{4}=   0.2013\cdots $.     
Let us consider  $t=0.3$ and realize diagonalizably the list $\Lambda  _{0.3}^{'0.5}=\{3.8, 2.7, -2,-2,-2.5\}$.
The characteristic polynomial of the list $\Gamma  _1^{'}=\{3.8, 2.7, -2,-2.5\}$ is $x^4 -2x^3-\frac{1399}{100}x^2+\frac{1367}{100}x+\frac{513}{10}$.
From the proof of Theorem \ref{CM} we obtain
$$
d_3=\frac{1399}{100}-d_1, \quad b=2d_1-\frac{1367}{100}, \quad a=-d_1^2+\frac{1799}{100}d_1-\frac{1966}{25}.
$$
The entries $d_3$ and $b$ are nonnegative for  $\frac{1367}{200} \leq d_1 \leq \frac{1399}{100}$. 
The entry $a$ is nonnegative for 
$d_1\in [\frac{1799-9\sqrt{1121}}{200}, \frac{1799+9\sqrt{1121}}{200}]=[7.483\cdots , 10.501\cdots ]$. 
Then the rank of $a$ is between $0$ and its maximum value attained in  $d_1=\frac{1799}{200}$, i.e.,
$a\in [0,2.270025]$.
If we take $d_1=9$ we obtain the matrices 

 $A(9)=\left[
\begin{array}{cccc}
  0 & 1 & 0   &  0    \\
  9 & 0 & 1   &  0   \\
 4.33  & 0 & 0   &  1    \\
 2.27  & 0 & 4.99 &  2  
\end{array}
\right]$ \; and \;
$C(9)=\left[
\begin{array}{ccccc}
  0   &  1  &  0   &  0  &  0  \\
  9   &  0  &  1   &  0  &  0  \\
 4.33   &  0  &  0   & 0.5 & 0.5    \\
 2.27 &  0  & 4.99 &  0  &  2    \\
 2.27 &  0  & 4.99 &  2  &  0   
\end{array}
\right]$ 

\noindent that realize $\Gamma _1^{'}$ and $\Lambda  _{0.3}^{'0.5}$ respectively.

\end{example}

\end{document}